\documentclass{amsproc}
\usepackage{amssymb,amsmath,enumerate,amscd,mathrsfs}

\numberwithin{equation}{section}

\newtheorem{theorem}[equation]{Theorem}

\newtheorem{lemma}[equation]{Lemma}
\newtheorem{corollary}[equation]{Corollary}

\theoremstyle{definition}
\newtheorem{definition}[equation]{Definition}

\newtheorem{remark}[equation]{Remark}

\DeclareMathOperator{\sym}{ \sigma\!\!\!\sigma}

\def\A{\mathscr A}
\def\B{\mathscr B}
\def\C{\mathbb C}

\def\Z{\mathbb Z}

\def\L{\mathscr L}
\def\N{\mathbb N}
\def\R{\mathbb R}
\def\P{\mathscr P}

\def\S{\mathscr S}
\def\U{\mathscr U}

\def\Mbar{\overline{M}}
\def\sing{\textup{sing}}
\def\Dom{\mathcal D}

\def\eps{\varepsilon}

\def\st{;\;}

\DeclareMathOperator{\Diff}{Diff}
\DeclareMathOperator{\supp}{supp}
\DeclareMathOperator{\ord}{ord}
\DeclareMathOperator{\Hom}{Hom}
\DeclareMathOperator{\End}{End}
\DeclareMathOperator{\Tr}{Tr}
\DeclareMathOperator{\tr}{tr}

\begin{document}
\title[Resolvent expansion for elliptic boundary contact 
problems]{On the expansion of the resolvent for elliptic 
boundary contact problems}
\author{Thomas Krainer}
\address{Penn State Altoona \\ 3000 Ivyside Park\\ Altoona, PA 16601\\ U.S.A.}
\email{krainer@psu.edu}

\begin{abstract}
Let $A$ be an elliptic operator on a compact manifold with boundary
$\Mbar$, and let $\wp : \partial\Mbar \to Y$ be a covering map,
where $Y$ is a closed manifold.
Let $A_C$ be a realization of $A$ subject to a coupling condition
$C$ that is elliptic with parameter in the sector $\Lambda$.
By a coupling condition we mean a nonlocal boundary condition
that respects the covering structure of the boundary.

We prove that the resolvent trace $\Tr_{L^2} (A_C-\lambda)^{-N}$
for $N$ sufficiently large has a complete asymptotic expansion
as $|\lambda| \to \infty$, $\lambda \in \Lambda$. In particular,
the heat trace $\Tr_{L^2}e^{-tA_C}$ has a complete asymptotic 
expansion as $t \to 0^+$, and the $\zeta$-function
has a meromorphic extension to $\C$.
\end{abstract}


\maketitle


\section{Introduction}
\label{sec-Introduction}

\noindent
This paper deals with the pursuit of Seeley's program
\cite{SeeleyComplex,SeeleyResBVP} for elliptic operators
on singular spaces that are given by a compact smooth manifold
$\Mbar$ with boundary together with a prescribed gluing rule that 
identifies finitely many boundary points with each other. The 
spaces  under consideration include, in particular, 
quantum graphs (\cite{KostrykinSchrader,Kuchment}) and
${\mathbb Z}/k$-manifolds (\cite{FreedMelrose,Rosenberg}).

More precisely, following \cite{SavinSternin}, we assume
that the boundary of $\Mbar$ is equipped with a covering
$\wp : \partial\Mbar \to Y$. The base manifold $Y$ is closed, 
and we do not assume that it is connected. In particular,
$\wp$ may have a different number of sheets over each connected
component of $Y$. The singular space $M_{\sing}$ is
obtained by collapsing the fibres $\wp^{-1}\{y\}$ to $y$ for
every $y \in Y$.

A quantum graph represents a one-dimensional example of such a 
space: $\Mbar$ is a disjoint union of intervals --- the edges 
of the graph --- and $Y$ is the set of vertices. For $y \in Y$, 
the set $\wp^{-1}\{y\}$ consists of those endpoints of edges 
that are joined to form the vertex $y$.

Another example for the situation under study is
a disjoint union $\Mbar = \overline{N}_1 \cup \overline{N}_2$
of smooth compact manifolds $\overline{N}_j$ with 
the same (or diffeomorphic) boundary $Y = \partial\overline{N}_j$, 
$j=1,2$. We get a $2$-sheeted covering $\wp : 
\partial\Mbar \to Y$, and by collapsing the points in 
$\wp^{-1}\{y\}$ to $y$ for every $y \in Y$ the manifolds 
$\overline{N}_j$ are glued along their common boundary to give a 
closed manifold.
This is the setup for surgery in spectral theory, and $\Mbar$
is a particular example for a ${\mathbb Z}/2$-manifold.
The resulting space $M_{\sing}$ is nonsingular in this situation.

The elliptic operators to be considered on $\Mbar$ are subject to 
boundary conditions that respect the coupling of boundary points 
given by $\wp$. We will refer to these conditions as coupling
conditions in the sequel (they are called nonlocal boundary
value problems in \cite{SavinSternin}). It makes sense to
think of the realization of an elliptic operator subject to
a coupling condition as a boundary contact problem.
Examples are operators of Laplace-type with Kirchhoff or
$\delta$-type conditions on a quantum graph (see
\cite{KostrykinSchrader,Kuchment}), and, in the case of a 
${\mathbb Z}/2$-manifold, operators with transmission or transfer 
conditions as discussed in the mathematical physics literature,
see also \cite{GilkeyAsymptotic}.
One of the motivations for the present work is to contribute to
the theoretical underpinning of the heat equation method for
these and related problems, specifically as regards the treatment
of general elliptic operators of arbitrary order.

Let $E \to \Mbar$ be a vector bundle, and let
$A \in \Diff^m(\Mbar,E)$, $m > 0$, be a differential operator
with coefficients in $\End(E)$ (all operators and structures in 
this work are assumed to be smooth on $\Mbar$). Fix a Riemannian 
metric on $\Mbar$ and a Hermitian metric on $E$.
Our main result is the following theorem.

\begin{theorem}\label{ResolventExpansion-Intro}
Let $C$ be a coupling condition for $A$, and assume that the
boundary contact problem $(A,C)$ is elliptic with parameter in the 
closed sector $\Lambda \subset \C$ (see Section~\ref{sec-Setup} 
for details). Then the following holds:
\begin{enumerate}[a)]
\item The operator $A_C = A$ with domain
$$
\Dom(A_C) = \{u \in H^{\ord(A)}(\Mbar,E) \st Cu = 0\}
$$
is a closed operator in $L^2 = L^2(\Mbar,E)$.
\item For $\lambda \in \Lambda$ with $|\lambda| > 0$ sufficiently 
large the resolvent $(A_C - \lambda)^{-1} : L^2 \to \Dom(A_C)$ 
exists and satisfies the norm estimate
$$
\|(A_C - \lambda)^{-1}\|_{\L(L^2)} = O(|\lambda|^{-1})
$$
as $|\lambda| \to \infty$.
\item For $N > \dim\Mbar/\ord(A)$ the operator
$(A_C - \lambda)^{-N} : L^2 \to L^2$ is trace class, and for
any $\varphi \in C^{\infty}(\Mbar,\End(E))$ we have an asymptotic 
expansion \begin{equation}\label{Resolventtraceasymp}
\Tr \bigl(\varphi(A_C - \lambda)^{-N}\bigr)  \sim
|\lambda|^{-N}\sum_{j=0}^{\infty}
c_j(\hat{\lambda})|\lambda|^{\frac{\dim\Mbar-j}{\ord(A)}}
\quad \textup{as $|\lambda| \to \infty$},
\end{equation}
where $c_j = c_j(\varphi,N,A,C) \in
C^{\infty}({\mathbb S}^1\cap\Lambda)$, and
$\hat{\lambda} = \lambda/|\lambda|$.
\end{enumerate}
\end{theorem}

\noindent
By standard arguments (see \cite{GilkeyIndexTheory,Lesch,Shubin}) 
we get the following corollary from
Theorem~\ref{ResolventExpansion-Intro}.

\begin{corollary}\label{HeatTraceExpansion-Intro}
Let $(A,C)$ be elliptic with parameter in a closed sector of
the form $\Lambda = \{\lambda \in \C \st
|\arg(\lambda)| \geq \pi/2 - \eps\}$ for some $\eps > 0$.
Then the following holds:
\begin{enumerate}[a)]
\item The heat semigroup $e^{-tA_C} : L^2 \to L^2$ exists and
is of trace class for $t > 0$, and for every $\varphi \in 
C^{\infty}(\Mbar,\End(E))$
we have an asymptotic expansion
\begin{equation}\label{Heattraceasymp}
\Tr \bigl(\varphi e^{-tA_C}\bigr) \sim \sum_{j=0}^{\infty}
\alpha_j t^{\frac{j-\dim\Mbar}{\ord(A)}}
\quad \textup{as $t \to 0^+$}
\end{equation}
with certain heat invariants $\alpha_j = \alpha_j(\varphi,A,C)$.
\item If $A_C = A_C^* > 0$, then for every $\varphi \in
C^{\infty}(\Mbar,\End(E))$ the zeta function
$$
\zeta(s,\varphi,A_C) = \Tr\bigl(\varphi A_C^{-s}\bigr)
$$
extends to a meromorphic function on $\C$ with at most simple
poles at the points $(\dim\Mbar-j)/\ord(A)$, $j \in \N_0$,
and regular on $-\N_0$.
\item If $A_C = A_C^*$, then the asymptotics of the eigenvalues 
$\lambda_1 \leq \lambda_2  \leq \ldots$ of $A_C$ (counting 
multiplicities) is given by Weyl's law
\begin{equation}\label{WeylsLaw-Intro}
\lambda_k \sim \textup{Const}\cdot k^{\ord(A)/\dim\Mbar} \quad \textup{as $k \to \infty$}.
\end{equation}
\end{enumerate}
\end{corollary}

\noindent
The proof of Theorem~\ref{ResolventExpansion-Intro} relies
on reducing the nonlocal boundary contact problem $(A,C)$ near
$\partial\Mbar$ to the standard case of a local boundary value 
problem for a system that is associated with $A$ on
$Y\times[0,\eps)$. This utilizes the push-forward map $\wp_!$.
The expansion then follows by approximating the resolvent of
$A_C$ by a parametrix in a suitable pseudodifferential calculus
that is modelled on Boutet de Monvel's calculus.


\section{Coupling conditions and ellipticity}
\label{sec-Setup}

\noindent
Let $U(\partial\Mbar) \cong \partial\Mbar\times[0,\eps)$ be
a collar neighborhood of the boundary.
We extend the covering $\wp : \partial\Mbar \to Y$ to a
covering $\wp : U(\partial\Mbar) \to Y\times[0,\eps)$ in
the obvious manner. Choose a Riemannian metric $g_Y$ on $Y$,
and consider the metric $h=g_Y + dx^2$ on $Y\times[0,\eps)$.
The given metric $g$ on $\Mbar$ and this choice of metric on
$Y\times[0,\eps)$ determine a (discrete) measure $\mu_{(y,x)}$
on the fibre $\wp^{-1}\{(y,x)\}$ for every $y \in Y$ and
$0 \leq x < \eps$ so that the canonical map
$$
L_g^2(U(\partial\Mbar)) \cong L_h^2(Y\times[0,\eps),\wp_!\C)
$$
induced by $\wp$ is unitary. Here $\C \to U(\partial\Mbar)$ denotes
the trivial line bundle, and $\wp_!\C \to Y\times[0,\eps)$ is
the vector bundle with fibre $\wp_!\C_{(y,x)} = 
L^2(\wp^{-1}\{(y,x)\},\mu_{(y,x)})$ for every $y \in Y$ and
$x \in [0,\eps)$. Strictly speaking, $\wp_!\C$ is not 
necessarily a vector bundle since we do not assume that the
number of sheets of the covering $\wp : \partial\Mbar \to Y$
is the same over each connected component of $Y$, but this
is resolved by considering each component separately if necessary.
Likewise, consider the push-forward bundle $\wp_!E \to
Y\times[0,\eps)$. For the same reason $\wp_!E$ is
not necessarily a bundle over $Y\times[0,\eps)$, but its
restriction to $Y_0\times[0,\eps)$ is a vector bundle for
each connected component $Y_0$ of $Y$. The fibre over
$(y,x)$ is $\wp_!E_{(y,x)} = L^2(\wp^{-1}\{(y,x)\},E)$, and
the fibrewise $L^2$-inner product with respect to the measure 
$\mu_{(y,x)}$ on $\wp^{-1}\{(y,x)\}$ and the given Hermitian
metric on $E$ induces a Hermitian metric on the bundle $\wp_!E$.
With this data, the canonical map
\begin{equation}\label{CanMap}
\U : L_g^2(U(\partial\Mbar),E) \cong L_h^2(Y\times[0,\eps),\wp_!E)
\end{equation}
induced by $\wp$ is unitary. Moreover, $\U$ is
an isomorphism
$$
\U : H^s_{\textup{loc}}(U(\partial\Mbar),E) \cong
H^s_{\textup{loc}}(Y\times[0,\eps),\wp_!E)
$$
between the Sobolev spaces for all $s \in \R$.

Let $\Lambda \subset \C$ be a closed sector of the form
$\Lambda = \{re^{i\varphi} \st r \geq 0, |\varphi-\varphi_0|
\leq a\}$ for some $a > 0$. Let $A \in \Diff^m(\Mbar,E)$, $m > 0$.
Our standing assumption is that $A$ is elliptic with parameter
in $\Lambda$. Recall that this means that the principal symbol
$$
\sym(A) \in C^{\infty}\bigl(T^*\Mbar\setminus 0, \End(\pi^*E)\bigr),
\textup{ where } \pi : T^*\Mbar \to \Mbar,
$$
has no eigenvalue in $\Lambda$.

\begin{definition}\label{BoundaryContactProblem}
Let $Y_0$ be any connected component of $Y$, and let $F_{0,j} \to
Y_0\times[0,\eps)$, $j=1,\ldots,M_0$, be vector bundles.
Let $B_{0,j} \in \Diff^{m_{0,j}}(Y_0\times[0,\eps),\wp_!E,F_{0,j})$
be differential operators, where $m_{0,j} < m$. We call the mapping
$$
C_0 = \gamma_{Y_0}\begin{pmatrix} B_{0,1} \\ \vdots \\ B_{0,M_0}
\end{pmatrix}\U \circ r_{U_0} : C^{\infty}(\Mbar,E) \to
C^{\infty}\Bigl(Y_0,\bigoplus_{j=1}^{M_0}F_{0,j}|_{Y_0\times\{0\}}\Bigr)
$$
a coupling condition associated with $Y_0$, where
$\gamma_{Y_0} : f \mapsto f|_{Y_0\times\{0\}}$ is the trace map 
for functions on $Y_0\times[0,\eps)$, and $r_{U_0} : 
C^{\infty}(\Mbar,E) \to C^{\infty}(U_0,E)$ is the 
restriction of functions to the subset
$U_0 = \wp^{-1}\bigl(Y_0\times[0,\eps)\bigr)$ of
the collar neighborhood of $\partial\Mbar$.
By a coupling condition we mean a map
$$
C : C^{\infty}(\Mbar,E) \to \bigoplus_{Y_0 \subset Y}
C^{\infty}\Bigl(Y_0,\bigoplus_{j=1}^{M_0}F_{0,j}|_{Y_0\times\{0\}}
\Bigr)
$$
given by a choice of coupling condition for each component $Y_0$.

The mapping
$$
\begin{pmatrix} A \\ C \end{pmatrix} : H^s(\Mbar,E) \to
\begin{array}{c} H^{s-m}(\Mbar,E) \\ \oplus \\
\bigoplus_{Y_0 \subset Y}\bigoplus_{j=1}^{M_0}
H^{s-m_{0,j}-1/2}(Y_0,F_{0,j}|_{Y_0\times\{0\}}) \end{array}
$$
is continuous for all $s > m-1/2$. We will refer to the pair $(A,C)$
as a boundary contact problem.
\end{definition}

Consider the operator
$$
\A = \U A \U^{-1} : C^{\infty}(Y\times[0,\eps),\wp_!E) \to
C^{\infty}(Y\times[0,\eps),\wp_!E).
$$
$\A \in \Diff^m(Y\times[0,\eps),\wp_!E)$, and $\A$ is elliptic
with parameter in $\Lambda$ since this is the case for $A$.
Locally, $\A$ can be regarded as a `diagonal operator' with
the various restrictions of $A$ to the sheets of $\wp$ on the
diagonal.

Let $\B_0 = \gamma_{Y_0}\begin{pmatrix} B_{0,1} & \cdots & B_{0,M_0}
\end{pmatrix}^{\textup{tr}}$. Then
\begin{equation}\label{BVPPushforward}
\begin{pmatrix} \A \\ \B_0 \end{pmatrix} :
C^{\infty}(Y_0\times[0,\eps),\wp_!E) \to
\begin{array}{c}
C^{\infty}(Y_0\times[0,\eps),\wp_!E) \\ \oplus \\
C^{\infty}\bigl(Y_0,\bigoplus_{j=1}^{M_0}F_{0,j}|_{Y_0\times\{0\}}\bigr)
\end{array}
\end{equation}
is a boundary value problem for $\A$ on $Y_0\times[0,\eps)$.

\begin{definition}\label{ElliptBoundaryContactProblem}
We call the boundary contact problem $(A,C)$ elliptic with
parameter in $\Lambda$ if $A$ is elliptic with parameter in
$\Lambda$, and if the boundary value problem
\eqref{BVPPushforward} is elliptic with parameter in $\Lambda$
for all connected components $Y_0 \subset Y$,
i.e., if $(\A,\B_0)$ satisfies the Agmon or parameter-dependent
Shapiro-Lopatinsky condition with respect to the sector $\Lambda$.
Recall that this means that the boundary symbol
$$
\begin{pmatrix} \sym(\A)(y,0,\eta,D_x) - \lambda \\
\gamma_{x=0} \sym(\B_0)(y,0,\eta,D_x)
\end{pmatrix} :
\S(\overline{\R}_+) \otimes \pi^*\wp_!E|_{Y_0\times\{0\}} \to
\begin{array}{c}
\S(\overline{\R}_+) \otimes \pi^*\wp_!E|_{Y_0\times\{0\}} \\
\oplus \\
\bigoplus_{j=1}^{M_0}\pi^*F_{0,j}|_{Y_0\times\{0\}}
\end{array}
$$
is invertible for all $(y,\eta;\lambda) \in
\bigl(T^*Y_0\times\Lambda\bigr)\setminus \{0\}$, where
$\pi : T^*Y_0 \to Y_0$ is the canonical projection
(see \cite{GrubbBuch,SzWiley98} for details).
Here $\sym(\A)(y,x,\eta,\xi)$ and $\sym(\B_0)(y,x,\eta,\xi)$
denote the (vectors) of homogeneous principal symbols of $\A$
and $\B_0$, respectively, and $\gamma_{x=0}$ is the evaluation
map $f \mapsto f(0)$ on $\S(\overline{\R}_+)$.
\end{definition}

Ellipticity without parameters and the Fredholm property for 
realizations of elliptic operators subject to coupling conditions 
has been addressed in \cite{SavinSternin}.


\section{A class of pseudodifferential operators}
\label{sec-PseudoClass}

\noindent
In this section we define an enveloping pseudodifferential 
calculus associated with boundary contact problems. This calculus
is modelled on Boutet de Monvel's algebra of pseudodifferential
boundary value problems that depend strongly polyhomogeneous
on a parameter $\lambda \in \Lambda$ (see
\cite{GrubbBuch,SzWiley98}). The resolvent of $A_C$ will be 
approximated by a parametrix in this calculus to furnish
the proof of Theorem~\ref{ResolventExpansion-Intro}, see
Section~\ref{sec-ParamAsymptotics}.

Let $J_{0,\pm} \to Y_0$ be vector bundles on each 
connected component $Y_0 \subset Y$ (the zero bundle is
admitted here). Fix a vector field
$\partial$ on the double $2\Mbar$ of $\Mbar$ that coincides in
the collar neighborhood $\partial\Mbar\times(-\eps,\eps)$ of the boundary
with the vector field $\partial_x$, and let
$$
\partial_+ = r_+\nabla^E_\partial e_+ : C^{\infty}(\Mbar,E) \to
C^{\infty}(\Mbar,E),
$$
where $e_+$ is the trivial extension operator by zero for functions
defined on $M$ to the double $2\Mbar$, $r_+$ is the 
restriction operator for distributions on $2\Mbar$ to $M$,
and $\nabla^E$ is a Hermitian connection on the (extended) bundle $E \to 2\Mbar$.

Moreover, let $\ell \in \N$ be fixed. $\ell$ represents the
anisotropy between the covariables and the parameter
$\lambda \in \Lambda$. For the treatment of the resolvent
of $A_C$ we will choose $\ell = m = \ord(A)$.

\begin{definition}[Regularizing Green operators]
\label{RegularizingGreen}
\begin{enumerate}[a)]
\item By $\Psi^{-\infty,0}(\Lambda)$ we denote the class of all 
operator families
$$
G(\lambda) : \begin{array}{c} H^s_0(\Mbar,E) \\ \oplus \\
\bigoplus\limits_{Y_0 \subset Y} H^s(Y_0,J_{0,-}) \end{array} \to
\begin{array}{c} H^t(\Mbar,E) \\ \oplus \\
\bigoplus\limits_{Y_0 \subset Y} H^t(Y_0,J_{0,+}) \end{array}
$$
that depend rapidly decreasing on $\lambda \in \Lambda$, for
all $s,t \in \R$.
In other words, $\Psi^{\infty,0}(\Lambda)$ consists of all
operator families with $C^{\infty}$-kernels that depend rapidly
decreasing on $\lambda \in \Lambda$.
\item For $d \in \N_0$ let $\Psi^{-\infty,d}(\Lambda)$ be
the class of all operator families of the form
$$
G(\lambda) = \sum\limits_{j=0}^d G_j(\lambda)
\begin{pmatrix} \partial_+ & 0 \\ 0 & 0 \end{pmatrix}^j :
\begin{array}{c} C^{\infty}(\Mbar,E) \\ \oplus \\
\bigoplus\limits_{Y_0 \subset Y} C^{\infty}(Y_0,J_{0,-}) 
\end{array} \to \begin{array}{c} C^{\infty}(\Mbar,E) \\ \oplus \\
\bigoplus\limits_{Y_0 \subset Y} C^{\infty}(Y_0,J_{0,+}) \end{array}
$$
with $G_j(\lambda) \in \Psi^{-\infty,0}(\Lambda)$.
\end{enumerate}
\end{definition}

Let $\varphi \in C^{\infty}(Y\times[0,\eps))$ be such that 
$\varphi$ is locally constant on $Y = Y\times\{0\}$. The 
restriction of $\varphi$ to $Y$ is a sum
\begin{equation}\label{PhiY}
\varphi|_Y = \sum\limits_{Y_0 \subset Y}a_{Y_0}\cdot\chi_{Y_0}
\end{equation}
of multiples of the characteristic functions $\chi_{Y_0}$
associated with the various connected components $Y_0$
of $Y$. Since multiplication by the characteristic function 
$\chi_{Y_1}$ is a projection operator, it thus makes sense to 
consider
$$
\chi_{Y_1} : \bigoplus\limits_{Y_0 \subset Y} C^{\infty}(Y_0,J_{0,\pm}) \to
C^{\infty}(Y_1,J_{1,\pm})
$$
as the projection operator to the subspace
$$
C^{\infty}(Y_1,J_{1,\pm}) \hookrightarrow
\bigoplus\limits_{Y_0 \subset Y} C^{\infty}(Y_0,J_{0,\pm}).
$$
Consequently, we consider
$$
\varphi|_Y :
\bigoplus\limits_{Y_0 \subset Y} C^{\infty}(Y_0,J_{0,\pm}) \to
\bigoplus\limits_{Y_0 \subset Y} C^{\infty}(Y_0,J_{0,\pm})
$$
an operator defined by the sum of multiples \eqref{PhiY}
of the projection operators associated with the characteristic 
functions of the connected components.

Let $\tilde\varphi \in C^{\infty}(\Mbar)$ be such that
$\tilde\varphi$ is the pull-back $\wp^*\varphi$ for a function
$\varphi$ as above near the boundary $\partial\Mbar$.
For such functions $\tilde\varphi$ we are going to use the 
notational convention that $\tilde\varphi$ is also to be 
understood as the operator
$$
\tilde\varphi = \begin{pmatrix} \tilde\varphi & 0 \\ 0 & \varphi|_Y
\end{pmatrix} :
\begin{array}{c} C^{\infty}(\Mbar,E) \\ \oplus \\
\bigoplus\limits_{Y_0 \subset Y} C^{\infty}(Y_0,J_{0,\pm}) 
\end{array} \to
\begin{array}{c} C^{\infty}(\Mbar,E) \\ \oplus \\
\bigoplus\limits_{Y_0 \subset Y} C^{\infty}(Y_0,J_{0,\pm}) 
\end{array},
$$
given by the multiplication operator with the function 
$\tilde\varphi$ in the upper left corner, and in
the lower right corner by the operator $\varphi|_Y$ explained
above.

\begin{definition}[Singular Green operators]
\label{SingularGreen}
Let $\mu \in \Z$, $d \in \N_0$. The class
$\Psi_G^{\mu,d}(\Lambda)$ consists of all operator families
$$
G(\lambda) : \begin{array}{c} C^{\infty}(\Mbar,E) \\ \oplus \\
\bigoplus\limits_{Y_0 \subset Y} C^{\infty}(Y_0,J_{0,-}) 
\end{array} \to \begin{array}{c} C^{\infty}(\Mbar,E) \\ \oplus \\
\bigoplus\limits_{Y_0 \subset Y} C^{\infty}(Y_0,J_{0,+}) \end{array}
$$
with the following properties:
\begin{itemize}
\item Let $Y_0 \subset Y$ be any connected component, and let
$\varphi,\psi \in C^{\infty}(Y_0\times[0,\eps))$ be compactly
supported and constant on $Y_0$. Consider the operator family
\begin{equation}\label{SingGreenComponent}
\begin{pmatrix} \U & 0 \\ 0 & 1 \end{pmatrix} \circ
\bigl(\wp^*\varphi G(\lambda) \wp^*\psi\bigr) \circ
\begin{pmatrix} \U^{-1} & 0 \\ 0 & 1 \end{pmatrix},
\end{equation}
acting in the spaces
$$
\begin{array}{c} C^{\infty}(Y_0\times[0,\eps),\wp_!E) \\ \oplus \\
C^{\infty}(Y_0,J_{0,-}) \end{array} \to
\begin{array}{c} C^{\infty}(Y_0\times[0,\eps),\wp_!E) \\ \oplus \\
C^{\infty}(Y_0,J_{0,+}) \end{array}.
$$
We require this family to belong to the class of
(strongly polyhomogeneous) anisotropic parameter-dependent (generalized)
singular Green operators of order $\mu$ and type $d$ in
Boutet de Monvel's calculus on $Y_0\times[0,\eps)$.
\item Let $\tilde\varphi,\tilde\psi \in C^{\infty}(\Mbar)$ be such
that $\tilde\varphi = \wp^*\varphi$ and $\tilde\psi = \wp^*\psi$
near the boundary $\partial\Mbar$ for functions $\varphi,\psi \in
C^{\infty}(Y\times[0,\eps))$ that are locally constant on
$Y$, and assume that $\supp\varphi \cap \supp\psi \cap Y =
\emptyset$.

We then require the operator family
$\tilde\varphi G(\lambda) \tilde\psi$ to belong to the class
$\Psi^{-\infty,d}(\Lambda)$ defined in
Definition~\ref{RegularizingGreen}.
\end{itemize}
Since the operator family
\eqref{SingGreenComponent} belongs to Boutet de Monvel's calculus
on $Y_0\times[0,\eps)$ it has a principal boundary symbol associated with it. Let
$\sym_{Y_0}(G)(y,\eta;\lambda)$ be that principal boundary symbol, an element of
$$
C^{\infty}\left(
\bigl(T^*Y_0\times\Lambda\bigr)\setminus \{0\},
\Hom\left(
\begin{array}{c}
\S(\overline{\R}_+)\otimes\pi^*\wp_!E\big|_{Y_0} \\
\oplus \\
\pi^*J_{0,-}
\end{array},
\begin{array}{c}
\S(\overline{\R}_+)\otimes\pi^*\wp_!E\big|_{Y_0} \\
\oplus \\
\pi^*J_{0,+}
\end{array}
\right)
\right),
$$
where $\pi : \bigl(T^*Y_0\times\Lambda\bigr)\setminus\{0\} \to Y_0$ is the canonical
projection.
Thus, associated with $G(\lambda)$, we have an operator valued principal $Y_0$-symbol on
$\bigl(T^*Y_0\times\Lambda\bigr)\setminus\{0\}$ for all connected components $Y_0 \subset Y$.
Recall that $\sym_{Y_0}(G)(y,\eta;\lambda)$ is homogeneous in the sense that
$\sym_{Y_0}(G)(y,\varrho\eta;\varrho^{\ell}\lambda)$ equals
$$
\varrho^{\mu}
\begin{pmatrix} \kappa_{\varrho}\otimes \textup{Id}_{\pi^*\wp_!E|_{Y_0}} & 0 \\ 0 & \textup{Id}_{\pi^*J_{0,+}} \end{pmatrix}
\sym_{Y_0}(G)(y,\eta;\lambda)
\begin{pmatrix} \kappa^{-1}_{\varrho}\otimes \textup{Id}_{\pi^*\wp_!E|_{Y_0}} & 0 \\ 0 & \textup{Id}_{\pi^*J_{0,-}} \end{pmatrix}
$$
for $\varrho > 0$, where $\bigl(\kappa_{\varrho}u\bigr)(x) = \varrho^{1/2}u(\varrho x)$.
\end{definition}

\begin{remark}\label{Boundarysymbolstructure}
In local coordinates near the boundary of $Y_0\times[0,\eps)$, the local boundary
symbols of the operator \eqref{SingGreenComponent} are operator families
$$
g(y,\eta;\lambda) = \sum\limits_{j=0}^d g_j(y,\eta;\lambda)\begin{pmatrix}
\partial_+ & 0 \\ 0 & 0 \end{pmatrix}^j,
$$
where the $g_j(y,\eta;\lambda)$ are boundary symbols of order $\mu-j$ and type zero.

A boundary symbol of order $\mu \in \Z$ and type zero is a $C^{\infty}$-function
\begin{equation}\label{BoundarySymbolMapping}
h(y,\eta;\lambda) : \begin{array}{c} \S'(\overline{\R}_+)\otimes \C^{\dim \wp_!E|_{Y_0}} \\
\oplus \\ \C^{\dim J_{0,-}} \end{array} \to
\begin{array}{c} \S(\overline{\R}_+)\otimes \C^{\dim \wp_!E|_{Y_0}} \\ \oplus \\
\C^{\dim J_{0,+}} \end{array}
\end{equation}
such that
$$
\begin{pmatrix} \kappa^{-1}_{(1+|\eta|+|\lambda|^{1/\ell})} & 0 \\ 0 & 1 \end{pmatrix}
\bigl(
\partial^{\alpha}_{y}\partial^{\beta}_{\eta}\partial^{\gamma}_{\lambda}h(y,\eta;\lambda)
\bigr)
\begin{pmatrix} \kappa_{(1+|\eta|+|\lambda|^{1/\ell})} & 0 \\ 0 & 1 \end{pmatrix}
$$
is $O\bigl((1+|\eta|+|\lambda|^{1/\ell})^{\mu-|\beta|-\ell|\gamma|}\bigr)$ as
$|(\eta,\lambda)| \to \infty$ in the topology of uniform convergence on bounded
subsets of the continuous operators in the spaces \eqref{BoundarySymbolMapping}, uniformly
for $y$ in compact subsets.
Moreover, $h(y,\eta;\lambda)$ has an asymptotic expansion
$$
h(y,\eta;\lambda) \sim \sum\limits_{j=0}^{\infty}\chi(\eta,\lambda)h_{(\mu-j)}(y,\eta;\lambda),
$$
where $\chi$ is an excision function of the origin, and the operator family
$h_{(\mu-j)}(y,\eta;\lambda)$ is (twisted) anisotropic homogeneous of degree
$\mu-j$ in the sense that
$$
h_{(\mu-j)}(y,\varrho\eta;\varrho^{\ell}\lambda) = \varrho^{\mu-j}
\begin{pmatrix} \kappa_{\varrho} & 0 \\ 0 & 1 \end{pmatrix}
h_{(\mu-j)}(y,\eta;\lambda)
\begin{pmatrix} \kappa_{\varrho}^{-1} & 0 \\ 0 & 1 \end{pmatrix}
$$
for $\varrho > 0$ and $(\eta,\lambda) \neq (0,0)$.

This description of the boundary symbol structure of generalized singular Green operators
in Boutet  de Monvel's calculus follows Schulze \cite{SzWiley98}, see also
\cite{SchroheShortIntro}.

Let
$$
h_0(y,\eta;\lambda) :
\S'(\overline{\R}_+)\otimes \C^{\dim \wp_!E|_{Y_0}} \to
\S(\overline{\R}_+)\otimes \C^{\dim \wp_!E|_{Y_0}}
$$
be the upper left corner of the symbol \eqref{BoundarySymbolMapping}. From the latter
description it is immediately clear that $h_0(y,\eta;\lambda)$ is of trace class as
an operator acting in $L^2(\overline{\R}_+)\otimes \C^{\dim \wp_!E|_{Y_0}}$, and that its
trace $\tr_{L^2}h_0(y,\eta;\lambda)$ is an ordinary anisotropic parameter-dependent
classical symbol of order $\mu$ (see Remark~\ref{OrdinarySymbols}).
\end{remark}

\begin{remark}\label{Relationsmoothing}
Let $\B_G^{\mu,d}(\Lambda)$ be the class of anisotropic parameter-dependent singular
Green operators $G(\lambda) : C^{\infty}(\Mbar,E) \to C^{\infty}(\Mbar,E)$ of order $\mu$ 
and type $d$ in Boutet de Monvel's calculus on $\Mbar$, and let $\B^{-\infty,d}(\Lambda)$
be the subspace of regularizing singular Green operators of type $d$. Then
$$
\B_G^{\mu,d}(\Lambda) \subset \Psi_G^{\mu,d}(\Lambda) \textup{ and }
\B^{-\infty,d}(\Lambda) = \Psi^{-\infty,d}(\Lambda).
$$
In the interesting cases for us $\B_G^{\mu,d}(\Lambda) \neq \Psi_G^{\mu,d}(\Lambda)$
because the covering $\wp : \partial\Mbar \to Y$ has multiple sheets.
\end{remark}

\begin{definition}[The full calculus]\label{FullAlgebra}
Let $\mu \in \Z$, $d \in \N_0$. The class $\Psi^{\mu,d}(\Lambda)$ consists of all
operator families
\begin{equation}\label{FullAlgebraOperator}
A(\lambda) =
\begin{pmatrix} r_+A_0(\lambda)e_+ & 0 \\ 0 & 0 \end{pmatrix} + G(\lambda) :
\begin{array}{c} C^{\infty}(\Mbar,E) \\ \oplus \\
\bigoplus\limits_{Y_0 \subset Y} C^{\infty}(Y_0,J_{0,-}) 
\end{array} \to \begin{array}{c} C^{\infty}(\Mbar,E) \\ \oplus \\
\bigoplus\limits_{Y_0 \subset Y} C^{\infty}(Y_0,J_{0,+}) \end{array},
\end{equation}
where $G(\lambda) \in \Psi^{\mu,d}_G(\Lambda)$, and $A_0(\lambda)$ is an anisotropic
parameter-dependent pseudodifferential operator on $2\Mbar$ with the transmission
property at $\partial\Mbar$.
\end{definition}

\begin{remark}\label{OrdinarySymbols}
The local symbols $a(z,\zeta;\lambda)$ of $A_0(\lambda)$ in
Definition~\ref{FullAlgebra} satisfy the symbol estimates
$$
|\partial^{\alpha}_z\partial^{\beta}_{\zeta}\partial^{\gamma}_{\lambda}a(z,\zeta;\lambda)|
= O\bigl((1 + |\zeta| + |\lambda|^{1/\ell})^{\mu-|\beta|-\ell|\gamma|}\bigr)
$$
as $|(\zeta,\lambda)| \to \infty$, uniformly for $z$ in compact subsets, and they have
an asymptotic expansion
$$
a(z,\zeta;\lambda) \sim \sum\limits_{j=0}^{\infty}
\chi(\zeta,\lambda)a_{(\mu-j)}(z,\zeta;\lambda),
$$
where $\chi$ is an excision function of the origin, and $a_{(\mu-j)}(z,\zeta;\lambda)$
is anisotropic homogeneous of degree $\mu-j$, i.e.,
$$
a_{(\mu-j)}(z,\varrho\zeta;\varrho^{\ell}\lambda) = \varrho^{\mu-j}
a_{(\mu-j)}(z,\zeta;\lambda)
\textup{ for $\varrho > 0$ and $(\zeta,\lambda) \neq (0,0)$}.
$$
\end{remark}

Let $\varphi \in C^{\infty}(Y_0\times[0,\eps))$ be compactly supported such that
$\varphi \equiv 1$ near $Y_0$.
$$
\U \circ \bigl(\wp^*\varphi r_+A_0(\lambda)e_+ \wp^*\varphi\bigr) \circ \U^{-1} :
C^{\infty}(Y_0\times[0,\eps),\wp_!E) \to C^{\infty}(Y_0\times[0,\eps),\wp_!E)
$$
is a parameter-dependent pseudodifferential operator in Boutet de Monvel's calculus
on $Y_0\times[0,\eps)$, and thus has a principal boundary symbol. For each connected
component $Y_0 \subset Y$, let $\sym_{Y_0}(A_0)(y,\eta;\lambda)$ be that principal boundary
symbol on $\bigl(T^*Y_0\times\Lambda)\setminus\{0\}$.
Consequently, the following principal symbols are associated with every operator
$A(\lambda) \in \Psi^{\mu,d}(\Lambda)$ as given by \eqref{FullAlgebraOperator}:
\begin{itemize}
\item The homogeneous principal symbol
$$
\sym(A)(z,\zeta;\lambda):= \sym(A_0)(z,\zeta;\lambda) \in
C^{\infty}\bigl((T^*\Mbar\times\Lambda)\setminus\{0\},\End(\pi^*E)\bigr),
$$
where $\pi : (T^*\Mbar\times\Lambda)\setminus\{0\} \to \Mbar$ is the canonical projection.
\item The principal $Y_0$-symbol
$$
\sym_{Y_0}(A)(y,\eta;\lambda):= \begin{pmatrix} \sym_{Y_0}(A_0)(y,\eta;\lambda) & 0 \\
0 & 0 \end{pmatrix} + \sym_{Y_0}(G)(y,\eta;\lambda)
$$
defined on $\bigl(T^*Y_0\times\Lambda\bigr)\setminus\{0\}$ for every connected component
$Y_0 \subset Y$.
\end{itemize}

\begin{definition}\label{EllipticityCalculus}
An operator family $A(\lambda) \in \Psi^{\mu,d}(\Lambda)$ is parameter-dependent
elliptic if its homogeneous principal symbol $\sym(A)(z,\zeta;\lambda)$ is
invertible for every $(z,\zeta;\lambda) \in (T^*\Mbar\times\Lambda)\setminus\{0\}$,
and its principal $Y_0$-symbol $\sym_{Y_0}(A)(y,\eta;\lambda)$ is invertible
for all $(y,\eta;\lambda) \in (T^*Y_0\times\Lambda)\setminus\{0\}$, for all
connected components $Y_0 \subset Y$.
\end{definition}

\begin{theorem}\label{ExtensiontoSobolevSpaces}
\begin{enumerate}[a)]
\item Every $A(\lambda) \in \Psi^{\mu,d}(\Lambda)$ extends by continuity to a family of
continuous operators
\begin{equation}\label{AinSobSp}
A(\lambda) : \begin{array}{c} H^s(\Mbar,E) \\ \oplus \\
\bigoplus\limits_{Y_0 \subset Y} H^s(Y_0,J_{0,-}) \end{array} \to
\begin{array}{c} H^{s-\mu}(\Mbar,E) \\ \oplus \\
\bigoplus\limits_{Y_0 \subset Y} H^{s-\mu}(Y_0,J_{0,+}) \end{array}
\end{equation}
for $s > d-1/2$.
\item Let $A(\lambda) \in \Psi^{\mu,0}(\Lambda)$, where $\mu \leq 0$.
Then the operator norm of
$$
A(\lambda) : \begin{array}{c} L^2(\Mbar,E) \\ \oplus \\
\bigoplus\limits_{Y_0 \subset Y} L^2(Y_0,J_{0,-}) \end{array} \to
\begin{array}{c} L^2(\Mbar,E) \\ \oplus \\
\bigoplus\limits_{Y_0 \subset Y} L^2(Y_0,J_{0,+}) \end{array}
$$
is $O(|\lambda|^{\mu/\ell})$ as $|\lambda| \to \infty$.
\end{enumerate}
\end{theorem}
\begin{proof}
Write
$$
A(\lambda) =
\begin{pmatrix} r_+A_0(\lambda)e_+ & 0 \\ 0 & 0 \end{pmatrix} + G(\lambda)
$$
as in \eqref{FullAlgebraOperator}. Both a) and b) are clear for
$\begin{pmatrix} r_+A_0(\lambda)e_+ & 0 \\ 0 & 0 \end{pmatrix}$.

Let $\varphi_0,\psi_0 \in C^{\infty}(Y_0\times[0,\eps))$ be compactly supported
such that $\varphi_0 \equiv 1$ near $Y_0$, and $\psi_0 \equiv 1$ in a neighborhood
of the support of $\varphi_0$. Then
$$
G(\lambda) = \sum\limits_{Y_0\subset Y}\wp^*\varphi_0 G(\lambda) \wp^*\psi_0 + R(\lambda),
$$
where $R(\lambda) \in \Psi^{-\infty,d}(\Lambda)$.
Both a) and b) are evident for $R(\lambda)$. The operator
$$
\begin{pmatrix} \U & 0 \\ 0 & 1 \end{pmatrix} \circ
\bigl(\wp^*\varphi_0 G(\lambda) \wp^*\psi_0\bigr)
\circ \begin{pmatrix} \U^{-1} & 0 \\ 0 & 1 \end{pmatrix}
$$
is a parameter-dependent generalized singular Green operator in Boutet de Monvel's
calculus on $Y_0\times[0,\eps)$ (supported near the boundary $Y_0$). Consequently,
both assertions a) and b) are valid for this operator on $Y_0\times[0,\eps)$.
Since the canonical map $\U$ is an isometry in $L^2$ and an isomorphism between the
Sobolev spaces, see the discussion around \eqref{CanMap}, a) and b) follow for
the operators $\wp^*\varphi_0 G(\lambda) \wp^*\psi_0$.
\end{proof}

\begin{theorem}\label{CompositionTheorem}
Let $A_j(\lambda) \in \Psi^{\mu_j,d_j}(\Lambda)$, $j=1,2$, and assume that the
vector bundles fit together such that the composition $A_1(\lambda)A_2(\lambda)$ is
defined.

Then $A_1(\lambda)A_2(\lambda) \in \Psi^{\mu_1+\mu_2,d}(\Lambda)$, where
$d = \max\{d_1+\mu_2,d_2\}$. We have
\begin{align*}
\sym(A_1A_2)(z,\zeta;\lambda) &= \sym(A_1)(z,\zeta;\lambda)\sym(A_2)(z,\zeta;\lambda), \\
\sym_{Y_0}(A_1A_2)(y,\eta;\lambda) &= \sym_{Y_0}(A_1)(y,\eta;\lambda)\sym_{Y_0}(A_2)(y,\eta;\lambda)
\end{align*}
for all connected components $Y_0 \subset Y$.
\end{theorem}
\begin{proof}
We first observe that the composition of operator families is well-defined in
\begin{equation}\label{RegularizingIdeal}
\begin{aligned}
\Psi^{\mu_1,d_1}(\Lambda)\times\Psi^{-\infty,d_2}(\Lambda) &\to \Psi^{-\infty,d}(\Lambda), \\
\Psi^{-\infty,d_1}(\Lambda)\times\Psi^{\mu_2,d_2}(\Lambda) &\to
\Psi^{-\infty,d}(\Lambda).
\end{aligned}
\end{equation}
The first of these statements follows immediately from the definition of
the calculus and Theorem~\ref{ExtensiontoSobolevSpaces}.
To prove the second, let $\varphi_0, \psi_0 \in C^{\infty}(Y_0\times[0,\eps))$ be
compactly supported, and let $\varphi_0 \equiv 1$ near $Y_0$, and $\psi_0 \equiv 1$ in a
neighborhood of the support of $\varphi_0$. Let
$$
\phi_{\textup{int}} = 1 - \sum\limits_{Y_0 \subset Y}\wp^*\varphi_0,
$$
and let $\psi_{\textup{int}} \in C^{\infty}(\Mbar)$ be compactly supported away from
$\partial\Mbar$ with $\psi_{\textup{int}} \equiv 1$ in a neighborhood of the support of
$\varphi_{\textup{int}}$.
Let $A(\lambda) \in \Psi^{\mu_2,d_2}(\Lambda)$. We write
\begin{equation}\label{Adecomp}
A(\lambda) = \sum\limits_{Y_0 \subset Y}\wp^*\varphi_0 A(\lambda) \wp^*\psi_0 +
\varphi_{\textup{int}} A(\lambda) \psi_{\textup{int}} + R(\lambda)
\end{equation}
with $R(\lambda) \in \Psi^{-\infty,d_2}(\Lambda)$.
Thus it remains to treat the composition of $G(\lambda) \in \Psi^{-\infty,d_1}(\Lambda)$
with every summand in \eqref{Adecomp}. Write
$$
G(\lambda) = \sum\limits_{j=0}^{d_1} G_j(\lambda)\begin{pmatrix} \partial_+ & 0 \\ 0 & 0
\end{pmatrix}^j
\textup{ with } G_j(\lambda) \in \Psi^{-\infty,0}(\Lambda).
$$
From the composition theorem in Boutet de Monvel's calculus on $Y_0\times[0,\eps)$
we obtain that
$$
\begin{pmatrix} \U & 0 \\ 0 & 1 \end{pmatrix}
\begin{pmatrix} \partial_+ & 0 \\ 0 & 0 \end{pmatrix}^j
\begin{pmatrix} \U^{-1} & 0 \\ 0 & 1 \end{pmatrix} \circ
\begin{pmatrix} \U & 0 \\ 0 & 1 \end{pmatrix}
\wp^*\varphi_0 A(\lambda) \wp^*\psi_0
\begin{pmatrix} \U^{-1} & 0 \\ 0 & 1 \end{pmatrix}
$$
belongs to Boutet de Monvel's calculus on $Y_0\times[0,\eps)$.
Since $G_j(\lambda)$ is an integral operator with $C^{\infty}$-kernel that
depends rapidly decreasing on $\lambda \in \Lambda$, we thus conclude that
$$
G_j(\lambda)\begin{pmatrix} \partial_+ & 0 \\ 0 & 0 \end{pmatrix}^j\wp^*\varphi_0 
A(\lambda) \wp^*\psi_0 \in \Psi^{-\infty,d}(\Lambda).
$$
The latter conclusion utilizes the mapping properties of $\U$, and the standard
mapping properties of operators in Boutet de Monvel's calculus. It is clear that
the compositions
$G(\lambda)\varphi_{\textup{int}} A(\lambda) \psi_{\textup{int}}$
and $G(\lambda)R(\lambda)$ belong to $\Psi^{-\infty,d}(\Lambda)$.
This finishes the proof of \eqref{RegularizingIdeal}.

Now consider the general case. Let $\tilde{\psi}_0 \in C^{\infty}(Y_0\times[0,\eps))$
be compactly supported with $\tilde{\psi}_0 \equiv 1$ in a neighborhood of the
support of $\psi_0$. Write
\begin{align*}
\wp^*\varphi_0 A_1(\lambda)A_2(\lambda) \wp^*\psi_0 =
\bigl(\wp^*\varphi_0 &A_1(\lambda)\wp^*\psi_0\bigr)
\bigl(\wp^*\tilde{\psi_0}A_2(\lambda)\wp^*\psi_0\bigr) \\
&+ \wp^*\varphi_0 A_1(\lambda)(1-\wp^*\tilde{\psi}_0)A_2(\lambda)\wp^*\psi_0.
\end{align*}
The operator $(1-\wp^*\tilde{\psi}_0)A_2(\lambda)\wp^*\psi_0 \in
\Psi^{-\infty,d_2}(\Lambda)$, and thus
$$
\wp^*\varphi_0 A_1(\lambda)(1-\wp^*\tilde{\psi}_0)A_2(\lambda)\wp^*\psi_0 \in
\Psi^{-\infty,d}(\Lambda)
$$
by \eqref{RegularizingIdeal}. Utilizing the mapping $\U$ and the
composition theorem in Boutet de Monvel's calculus on $Y_0\times[0,\eps)$, we get
$$
\bigl(\wp^*\varphi_0 A_1(\lambda)\wp^*\psi_0\bigr)
\bigl(\wp^*\tilde{\psi_0}A_2(\lambda)\wp^*\psi_0\bigr) \in
\Psi^{\mu_1+\mu_2,d}(\Lambda).
$$
Hence $\wp^*\varphi_0 A_1(\lambda)A_2(\lambda) \wp^*\psi_0 \in
\Psi^{\mu_1+\mu_2,d}(\Lambda)$, and
\begin{align*}
\sym(\wp^*\varphi_0 A_1(\lambda)A_2(\lambda) \wp^*\psi_0) & =
\wp^*\varphi_0 \sym(A_1)\sym(A_2), \\
\sym_{Y_0}(\wp^*\varphi_0 A_1(\lambda)A_2(\lambda) \wp^*\psi_0) &=
\sym_{Y_0}(A_1A_2) = \sym_{Y_0}(A_1)\sym_{Y_0}(A_2).
\end{align*}
One comment about this argument is in order. Decompose the $A_j(\lambda)$ according to
\eqref{FullAlgebraOperator} into pseudodifferential and singular Green
parts. The pseudodifferential parts multiply by the composition theorem in
Boutet de Monvel's calculus on $\Mbar$. Hence we have the pseudodifferential parts of
$\wp^*\varphi_0 A_1(\lambda)A_2(\lambda) \wp^*\psi_0$ under control, we don't catch 
contributions that are not pseudolocal on $\Mbar$ by pulling back pseudodifferential 
operators from $Y_0\times[0,\eps)$ via $\U$.

Let $\hat{\psi_0} \in C^{\infty}(Y_0\times[0,\eps))$ be compactly supported such
that $\hat{\psi_0} \equiv 1$ in a neighborhood of the support of $\varphi_0$, and
such that $\psi_0 \equiv 1$ in a neighborhood of the support of $\hat{\psi}_0$.
Then
\begin{align*}
\wp^*\varphi_0 A_1(\lambda)A_2(\lambda)(1 - \wp^*\psi_0) =
\bigl(\wp^*\varphi_0 &A_1(\lambda) (1-\wp^*\hat{\psi}_0)\bigr)
A_2(\lambda)(1 - \wp^*\psi_0) \\
&+ \wp^*\varphi_0 A_1(\lambda) \bigl(\wp^*\hat{\psi}_0A_2(\lambda)(1 - \wp^*\psi_0)\bigr).
\end{align*}
Now
$$
\bigl(\wp^*\varphi_0 A_1(\lambda) (1-\wp^*\hat{\psi}_0)\bigr) \in
\Psi^{-\infty,d_1}(\Lambda) \textup{ and }
\bigl(\wp^*\hat{\psi}_0A_2(\lambda)(1 - \wp^*\psi_0)\bigr) \in
\Psi^{-\infty,d_2}(\Lambda),
$$
and thus
$$
\wp^*\varphi_0 A_1(\lambda)A_2(\lambda)(1 - \wp^*\psi_0) \in \Psi^{-\infty,d}(\Lambda)
$$
by \eqref{RegularizingIdeal}.
A similar argument shows that
$$
\varphi_{\textup{int}} A_1(\lambda)A_2(\lambda) \psi_{\textup{int}} \in
\Psi^{\mu_1+\mu_2,d}(\Lambda)
$$
with
$$
\sym(\varphi_{\textup{int}} A_1(\lambda)A_2(\lambda) \psi_{\textup{int}}) =
\varphi_{\textup{int}} \sym(A_1)\sym(A_2),
$$
and likewise
$$
\varphi_{\textup{int}} A_1(\lambda) A_2(\lambda) (1 - \psi_{\textup{int}}) \in
\Psi^{-\infty,d}(\Lambda).
$$
Now write
$$
A_1(\lambda)A_2(\lambda) = \sum\limits_{Y_0\subset Y}
\wp^*\varphi_0 A_1(\lambda)A_2(\lambda) \wp^*\psi_0 +
\varphi_{\textup{int}} A_1(\lambda)A_2(\lambda) \psi_{\textup{int}} + R(\lambda).
$$
According to the arguments given above we conclude that $R(\lambda) \in
\Psi^{-\infty,d}(\Lambda)$, and we see that every summand in this representation belongs to
$\Psi^{\mu_1+\mu_2,d}(\Lambda)$. Consequently,
$A_1(\lambda)A_2(\lambda) \in \Psi^{\mu_1+\mu_2,d}(\Lambda)$, and the asserted
identities for the principal symbols follow from the corresponding identities
obtained above for the summands. This finishes the proof of the theorem.
\end{proof}

\begin{theorem}\label{ParametrixTheorem}
Let $A(\lambda) \in \Psi^{\mu,d}(\Lambda)$ be parameter-dependent elliptic in the
sense of Definition~\ref{EllipticityCalculus}. Then there exists a parameter-dependent
parametrix $P(\lambda) \in \Psi^{-\mu,(d-\mu)_+}(\Lambda)$ of $A(\lambda)$,
where $(d-\mu)_+ = \max\{d-\mu,0\}$, i.e., we have
$$
P(\lambda)A(\lambda) - 1 \in \Psi^{-\infty,*}(\Lambda), \textup{ and }
A(\lambda)P(\lambda) - 1 \in \Psi^{-\infty,*}(\Lambda).
$$
The types of these regularizing remainders are given by the type formula from
Theorem~\ref{CompositionTheorem}.
\end{theorem}
\begin{proof}
Consider the restrictions of $A(\lambda)$ to $\wp^{-1}(Y_0\times[0,\eps))$, i.e.,
the operators
$$
A(\lambda)|_{\wp^{-1}(Y_0\times[0,\eps))} : \begin{array}{c} 
C_c^{\infty}(\wp^{-1}(Y_0\times[0,\eps)),E) \\ \oplus \\ C^{\infty}(Y_0,J_{0,-}) 
\end{array} \to
\begin{array}{c} C^{\infty}(\wp^{-1}(Y_0\times[0,\eps)),E) \\ \oplus \\
C^{\infty}(Y_0,J_{0,+}) \end{array}.
$$
By construction of the calculus and by assumption, we obtain that
$$
\begin{pmatrix} \U & 0 \\ 0 & 1 \end{pmatrix}
A(\lambda)|_{\wp^{-1}(Y_0\times[0,\eps))}
\begin{pmatrix} \U^{-1} & 0 \\ 0 & 1 \end{pmatrix}
$$
is a parameter-dependent elliptic boundary value problem in Boutet de Monvel's calculus
on $Y_0\times[0,\eps)$. Let $\P_{Y_0}(\lambda)$ be a parameter-dependent parametrix
of this operator, and consider
$$
\wp^*\varphi_0 \begin{pmatrix} \U^{-1} & 0 \\ 0 & 1 \end{pmatrix} \P_{Y_0}(\lambda)
\begin{pmatrix} \U & 0 \\ 0 & 1 \end{pmatrix} \wp^*\psi_0
$$
with the functions $\varphi_0,\psi_0$ from the proof of Theorem~\ref{CompositionTheorem}.
This operator belongs to $\Psi^{-\mu,(d-\mu)_+}(\Lambda)$.

Let $r_+P_0(\lambda)e_+$ be a parameter-dependent parametrix of the (interior)
pseudodifferential part $r_+A_0(\lambda)e_+$ of $A(\lambda)$ according to
\eqref{FullAlgebraOperator}, and let
$$
P(\lambda) = \sum\limits_{Y_0 \subset Y}
\wp^*\varphi_0 \begin{pmatrix} \U^{-1} & 0 \\ 0 & 1 \end{pmatrix} \P_{Y_0}(\lambda)
\begin{pmatrix} \U & 0 \\ 0 & 1 \end{pmatrix} \wp^*\psi_0 + \varphi_{\textup{int}}
\begin{pmatrix} r_+P_0(\lambda)e_+ & 0 \\ 0 & 0 \end{pmatrix}\psi_{\textup{int}},
$$
where $\varphi_{\textup{int}}$ and $\psi_{\textup{int}}$ are as in the proof of
Theorem~\ref{CompositionTheorem}.
By construction, $P(\lambda) \in \Psi^{-\mu,(d-\mu)_+}(\Lambda)$ is then
a parameter-dependent parametrix of $A(\lambda)$ as desired.
\end{proof}


\section{The expansion of the resolvent}
\label{sec-ParamAsymptotics}

\noindent
This last section is devoted to the proof of the main result
Theorem~\ref{ResolventExpansion-Intro}. We break it up into two parts,
Theorem~\ref{ResolventExists} and Theorem~\ref{ResolventExpansion} below.
Having the calculus from Section~\ref{sec-PseudoClass} at hand, we are able to
argue parallel to the classical arguments for closed manifolds and differential boundary
value problems, see \cite{GilkeyIndexTheory,GrubbBuch,GrubbWeaklyPoly}.
In what follows, we use the notation and conventions from Section~\ref{sec-Setup}.

Let
$$
C_0 = \gamma_{Y_0}\begin{pmatrix} B_{0,1} \\ \vdots \\ B_{0,M_0}
\end{pmatrix}\U \circ r_{U_0} : C^{\infty}(\Mbar,E) \to
C^{\infty}\Bigl(Y_0,\bigoplus_{j=1}^{M_0}F_{0,j}|_{Y_0\times\{0\}}\Bigr)
$$
be the coupling condition for $A$ associated with $Y_0$ from
Definition~\ref{BoundaryContactProblem}. Recall that the operators $B_{0,j}$ have
orders $m_{0,j} < m$. For each $j$, choose a family
$$
R_j(\lambda) : C^{\infty}(Y_0,F_{0,j}|_{Y_0\times\{0\}}) \to
C^{\infty}(Y_0,F_{0,j}|_{Y_0\times\{0\}})
$$
of order $m-(m_{0,j}-1/2)$ in the calculus of anisotropic parameter-dependent 
pseudodifferential operators on $Y_0$ that is invertible with inverse
$R_j(\lambda)^{-1}$ being of order $m_{0,j}-1/2-m$ in that calculus. Throughout
this section, the anisotropy is fixed to be $\ell = m = \ord(A)$.

Let
$$
J_{0,+} = \bigoplus_{j=1}^{M_0}F_{0,j}|_{Y_0\times\{0\}},
$$
and let
$$
T_0(\lambda) = \begin{pmatrix} R_1(\lambda) &  & 0 \\  & \ddots &  \\ 0 &  & R_{M_0}(\lambda)
\end{pmatrix} C_0 : C^{\infty}(\Mbar,E) \to C^{\infty}(Y_0,J_{0,+}).
$$
Let $T(\lambda)$ be the direct sum of the operators $T_0(\lambda)$. Then
\begin{equation}\label{OperatorInCalculus}
A(\lambda) = \begin{pmatrix} A - \lambda \\ T(\lambda) \end{pmatrix} :
C^{\infty}(\Mbar,E) \to \begin{array}{c} C^{\infty}(\Mbar,E) \\ \oplus \\
\bigoplus\limits_{Y_0 \subset Y}C^{\infty}(Y_0,J_{0,+}) \end{array}
\end{equation}
belongs to the operator class $\Psi^{m,m}(\Lambda)$ constructed in
Section~\ref{sec-PseudoClass}.
The following lemma is immediate.

\begin{lemma}\label{EllipticityCompare}
The boundary contact problem $(A,C)$ is elliptic with parameter in $\Lambda$ in the
sense of Definition~\ref{ElliptBoundaryContactProblem} if and only if the operator
family $A(\lambda) \in \Psi^{m,m}(\Lambda)$ from \eqref{OperatorInCalculus} is
parameter-dependent elliptic in the sense of Definition~\ref{EllipticityCalculus}.
\end{lemma}

\begin{theorem}\label{ResolventExists}
Let $(A,C)$ be elliptic with parameter in $\Lambda$. 
\begin{enumerate}[a)]
\item The operator
$$
\begin{pmatrix} A - \lambda \\ C \end{pmatrix} : H^s(\Mbar,E) \to
\begin{array}{c} H^{s-m}(\Mbar,E) \\ \oplus \\
\bigoplus_{Y_0 \subset Y}\bigoplus_{j=1}^{M_0}
H^{s-m_{0,j}-1/2}(Y_0,F_{0,j}|_{Y_0\times\{0\}}) \end{array}
$$
is invertible for all $s > m-1/2$ and all $\lambda \in \Lambda$ with $|\lambda|$
sufficiently large.

Consequently, the unbounded operator $A_C$ in $L^2(\Mbar,E)$ that acts like $A$ and
has domain
$$
\Dom(A_C) = \{u \in H^{m}(\Mbar,E) \st Cu = 0\}
$$
is closed and densely defined, and for large $\lambda \in \Lambda$ the
resolvent $(A_C - \lambda)^{-1}$ exists.
\item There exists $Q(\lambda) \in \Psi^{-m,0}(\lambda)$ such that
$(A_C - \lambda)^{-1} = Q(\lambda)$ for large $\lambda \in \Lambda$. In particular,
by Theorem~\ref{ExtensiontoSobolevSpaces}, the resolvent satisfies the norm estimate
$$
\|(A_C - \lambda)^{-1}\|_{\L(L^2(\Mbar,E))} = O(|\lambda|^{-1})
$$
as $|\lambda| \to \infty$.
\end{enumerate}
\end{theorem}
\begin{proof}
The operator $A(\lambda)$ in \eqref{OperatorInCalculus} is parameter-dependent
elliptic. Thus, by Theorem~\ref{ParametrixTheorem}, there exists a parametrix
$P(\lambda) \in \Psi^{-m,0}(\Lambda)$ of $A(\lambda)$. Consequently,
$$
A(\lambda)P(\lambda) - 1 = R_1(\lambda) \in \Psi^{-\infty,0}(\Lambda) \textup{ and }
P(\lambda)A(\lambda) - 1 = R_2(\lambda)\in \Psi^{-\infty,m}(\Lambda).
$$
For large $\lambda \in \Lambda$, the operators $1 + R_1(\lambda)$ and $1 + R_2(\lambda)$
are invertible as bounded operators in the Sobolev spaces. Write
$$
\bigl(1 + R_j(\lambda)\bigr)^{-1} = 1 - R_j(\lambda) + R_j(\lambda)
\chi(\lambda)\bigl(1 + R_j(\lambda)\bigr)^{-1}
R_j(\lambda)
$$
for large $\lambda$, where $\chi$ is a suitable excision function of the origin.
From the definition of the class of regularizing operators we obtain that
\begin{align*}
R_1(\lambda)
\chi(\lambda)\bigl(1 + R_1(\lambda)\bigr)^{-1}
R_1(\lambda) &\in \Psi^{-\infty,0}(\Lambda), \\
\intertext{and}
R_2(\lambda)
\chi(\lambda)\bigl(1 + R_2(\lambda)\bigr)^{-1}
R_2(\lambda) &\in \Psi^{-\infty,m}(\Lambda).
\end{align*}
Thus $\bigl(1 + R_j(\lambda)\bigr)^{-1} = 1 + \tilde{R}_j(\lambda)$ with regularizing
operators $\tilde{R}_j(\lambda)$. Consequently,
$$
P(\lambda)(1 + \tilde{R}_1(\lambda)) = \begin{pmatrix} Q(\lambda) & K(\lambda) \end{pmatrix}
\in \Psi^{-m,0}(\Lambda)
$$
inverts the operator $A(\lambda)$ for large $\lambda \in \Lambda$. Both a) and b) thus
follow in view of
$$
\Dom(A_C) = \{u \in H^{m}(\Mbar,E) \st T(\lambda)u = 0\}.
$$
\end{proof}

\begin{theorem}\label{ResolventExpansion}
Let $(A,C)$ be elliptic with parameter in the sector $\Lambda$, and let
$B \in \Diff^{k}(\Mbar,E)$.
For $N  > \frac{\dim\Mbar + k}{m}$ and large $\lambda \in \Lambda$ the operator
$$
B(A_C - \lambda)^{-N} : L^2(\Mbar,E) \to L^2(\Mbar,E)
$$
is of trace class, and the trace has an asymptotic expansion
$$
\Tr \bigl(B(A_C - \lambda)^{-N}\bigr) \sim |\lambda|^{-N}\sum\limits_{j=0}^{\infty}
c_j(\hat{\lambda})|\lambda|^{\frac{\dim\Mbar + k - j}{m}} \textup{ as }
|\lambda| \to \infty,
$$
where $c_j = c_j(B,N,A,C) \in C^{\infty}({\mathbb S}^1\cap\Lambda)$, and
$\hat{\lambda} = \lambda/|\lambda|$.
\end{theorem}
\begin{proof}
By Theorem~\ref{ResolventExists}, we have $(A_C - \lambda)^{-1} = Q(\lambda)$ for
large $\lambda \in \Lambda$ with $Q(\lambda) \in \Psi^{-m,0}(\Lambda)$. Because
$B \in \Psi^{k,0}(\lambda)$, we obtain from the composition theorem
(Theorem~\ref{CompositionTheorem}) that
$B(A_C - \lambda)^{-N} = BQ(\lambda)^{N} \in \Psi^{-Nm+k,0}(\Lambda)$.
Since the embedding $H^s(\Mbar,E) \hookrightarrow L^2(\Mbar,E)$ is nuclear for
$s > \dim\Mbar$, we get from Theorem~\ref{ExtensiontoSobolevSpaces} that
$$
B(A_C - \lambda)^{-N} : L^2(\Mbar,E) \to H^{Nm-k}(\Mbar,E) \hookrightarrow L^2(\Mbar,E)
$$
is of trace class as an operator in $L^2(\Mbar,E)$ provided that $Nm-k > \dim\Mbar$.
This proves the first assertion of the theorem.

In order to show the expansion, it suffices to show that for any
$P(\lambda) \in \Psi^{\mu,0}(\Lambda)$, where $\mu < -\dim\Mbar$, the $L^2$-trace
$\Tr P(\lambda)$ has an asymptotic expansion
$$
\Tr P(\lambda) \sim \sum\limits_{j=0}^{\infty}
c_j(\hat{\lambda})|\lambda|^{\frac{\dim\Mbar + \mu - j}{m}} \textup{ as }
|\lambda| \to \infty.
$$
To do this, we decompose $P(\lambda) = r_+P_0(\lambda)e_+ + G(\lambda)$ in a
pseudodifferential and singular Green operator according to \eqref{FullAlgebraOperator},
and consider the terms separately. The expansion of the trace of the pseudodifferential
part $\Tr(r_+P_0(\lambda)e_+)$ follows like in the case of a closed manifold.
For the benefit of the reader, we briefly sketch the argument:

Choose a partition of unity $\varphi_1,\ldots,\varphi_M$ subordinate to a finite
covering of $\Mbar$ by coordinate neighborhoods such that, in addition, $E$ is trivial
over each of these neighborhoods. Choose functions $\psi_j$ supported in the respective
coordinate neighborhoods such that $\psi_j \equiv 1$ in a
neighborhood of the support of $\varphi_j$. Write the operator
$P_0^+(\lambda) = r_+P_0(\lambda)e_+$ as
$$
P_0^+(\lambda) = \sum\limits_{j=0}^{M}\varphi_j P_0^+(\lambda)\psi_j + R(\lambda),
$$
where $R(\lambda) \in \Psi^{-\infty,0}(\Lambda)$. Because $\Tr R(\lambda) \sim 0$,
the expansion of $\Tr P_0^+(\lambda)$ reduces to expanding
$\Tr\bigl(\varphi_j P_0^+(\lambda)\psi_j\bigr)$ for each $j$. This converts to a
problem in coordinates. In coordinates, the $L^2$-trace is given by the
integral of the trace of the Schwartz kernel over the diagonal. Thus, if $p(z,\zeta;\lambda)$
is a local symbol, the trace is given by
$$
(2\pi)^{-n}\iint \tr p(z,\zeta;\lambda)\,dz\,d\zeta \sim
\sum\limits_{j=0}^{\infty}(2\pi)^{-n}\iint \tr p_{(\mu-j)}(z,\zeta;\lambda)\,dz\,d\zeta
\textup{ as } |\lambda| \to \infty,
$$
where $p_{(\mu-j)}(z,\zeta;\lambda)$ is the anisotropic homogeneous component of
degree $\mu-j$ of $p$, $n = \dim\Mbar$, and $\tr$ denotes the fibrewise trace. Recall that
$$
p_{(\mu-j)}(z,\varrho\zeta;\varrho^m\lambda) = \varrho^{\mu-j}p_{(\mu-j)}(z,\zeta;\lambda)
$$
for $\varrho > 0$ and $(\zeta,\lambda) \neq (0,0)$. Thus
\begin{align*}
(2\pi)^{-n}\iint \tr p_{(\mu-j)}&(z,\zeta;\lambda)\,dz\,d\zeta \\
&= \Bigl((2\pi)^{-n}\iint\tr p_{(\mu-j)}(z,|\lambda|^{-1/m}\zeta;\hat{\lambda})\,dz\,d\zeta\Bigr)
\cdot |\lambda|^{(\mu-j)/m} \\
&= \Bigl((2\pi)^{-n}\iint\tr p_{(\mu-j)}(z,\zeta;\hat{\lambda})\,dz\,d\zeta\Bigr) \cdot
|\lambda|^{(n+\mu-j)/m}.
\end{align*}
This shows that the desired asymptotic expansion holds for $\Tr P_0^+(\lambda)$.

Now consider the trace $\Tr G(\lambda)$. We shall work with the functions $\varphi_0$ and
$\psi_0$ from the proof of Theorem~\ref{CompositionTheorem} (these are not to be
confused with the $\varphi_j$'s and $\psi_j$'s above). Write
$$
G(\lambda) = \sum\limits_{Y_0 \subset Y}\varphi_0 G(\lambda) \psi_0 + \tilde{G}(\lambda)
$$
with $\tilde{G}(\lambda) \in \Psi^{-\infty,0}(\Lambda)$. In view of
$\Tr\tilde{G}(\lambda) \sim 0$, we only need to expand the trace
$\Tr\bigl(\varphi_0 G(\lambda) \psi_0\bigr)$ for each $Y_0$. From the trace property we
get
$$
\Tr\bigl(\varphi_0 G(\lambda) \psi_0\bigr) = \Tr\bigl(\U\circ
\bigl(\varphi_0 G(\lambda) \psi_0\bigr)\circ\U^{-1}\bigr)
$$
with the canonical map $\U$ from \eqref{CanMap}. The operator
$G_{Y_0}(\lambda) = \U\circ\bigl(\varphi_0 G(\lambda) \psi_0\bigr)\circ\U^{-1}$ is
an anisotropic parameter-dependent singular Green operator of order $\mu$ and type zero in
Boutet de Monvel's calculus on $Y_0\times[0,\eps)$ supported near the boundary. Hence
we know that
$$
\Tr G_{Y_0}(\lambda) \sim \sum\limits_{j=0}^{\infty}d_j(\hat{\lambda})
|\lambda|^{\frac{(n-1)+\mu-j}{m}} \textup{ as } |\lambda| \to \infty.
$$
For the benefit of the reader, let us stress that the latter expansion is
proved following the same scheme as the expansion of the pseudodifferential part above:

A similar localization argument as the one given above reduces the task of expanding
$\Tr G_{Y_0}(\lambda)$ to expanding the trace in coordinates near the boundary.
The boundary symbols $g(y,\eta;\lambda)$ of $G_{Y_0}(\lambda)$ have the structure
explained in Remark~\ref{Boundarysymbolstructure}. In coordinates, the $L^2$-trace
is given by
$$
(2\pi)^{-(n-1)}\iint \tr g(y,\eta;\lambda)\,dy\,d\eta,
$$
where $\tr$ denotes the trace on the space $L^2(\overline{\R}_+)
\otimes\C^{\dim\wp_!E|_{Y_0}}$. As noted in Remark~\ref{Boundarysymbolstructure},
$\tr g(y,\eta;\lambda)$ is an ordinary parameter-dependent symbol of order $\mu$.
The homogeneous components are the $\tr g_{(\mu-j)}(y,\eta;\lambda)$, where, for each
$j$, $g_{(\mu-j)}(y,\eta;\lambda)$ is the (twisted) homogeneous component of degree
$\mu-j$ associated with $g(y,\eta;\lambda)$.
Thus the same argument as above implies the expansion of the trace as desired.
\end{proof}



\begin{thebibliography}{10}



\bibitem{FreedMelrose}
D.~Freed and R.~Melrose, \emph{A mod $k$ index theorem},
Invent.~Math.~\textbf{107} (1992), no.~2, 283--299.

\bibitem{GilkeyIndexTheory}
P.~Gilkey, \emph{Invariance theory, the heat equation, and the
Atiyah-Singer index theorem}, 2nd ed., Studies in Advanced
Mathematics, CRC Press, Boca Raton, 1995.

\bibitem{GilkeyAsymptotic}
\bysame, \emph{Asymptotic formulae in spectral geometry},
Studies in Advanced Mathematics, Chapman \& Hall/CRC, Boca Raton,
2004.

\bibitem{GrubbBuch} 
G.~Grubb, \emph{Functional calculus of pseudodifferential boundary 
problems}, 2nd ed., Progress in Mathematics, vol.~65.
Birkh{\"a}user, Basel, 1996.

\bibitem{GrubbWeaklyPoly}
\bysame, \emph{A weakly polyhomogeneous calculus for
pseudodifferential boundary problems}, J.~Funct.~Anal.~\textbf{184}
(2001), 19--76.


\bibitem{KostrykinSchrader}
V.~Kostrykin and R.~Schrader, \emph{Laplacians on metric graphs:
Eigenvalues, resolvents, and semigroups}, in G.~Berkolaiko, R.~Carlson,
S.A.~Fulling, and P.~Kuchment (Eds.), \emph{Quantum Graphs and Their
Applications}, Contemp.~Math. Vol.~415, Amer. Math. Soc., Providence,
RI, 2006, pp.~201--225.

\bibitem{Kuchment}
P.~Kuchment, \emph{Quantum graphs I. Some basic structures}, Waves
Random Media \textbf{14} (2004), S107--S128.

\bibitem{Lesch}
M.~Lesch, \emph{Operators of Fuchs type, conical singularities, 
and asymptotic methods}, Teubner Texte zur Mathematik,
vol.~136, Teubner-Verlag, Leipzig, 1997.

\bibitem{Rosenberg}
J.~Rosenberg, \emph{Groupoid $C^*$-algebras and index theory
on manifolds with singularities}, Geom.~Dedicata \textbf{100}
(2003), 65--84.

\bibitem{SavinSternin}
A.~Savin and B.~Sternin, \emph{The index defect in the theory
of nonlocal problems and the $\eta$-invariant},
Sb.~Math.~\textbf{195} (2004), no.~9-10, 1321--1358.

\bibitem{SchroheShortIntro}
E.~Schrohe, \emph{A short introduction to Boutet de Monvel's calculus},
Oper.~Theory Adv.~Appl. \textbf{125} (2001), 85--116.

\bibitem{SzWiley98}
B.-W.~Schulze, \emph{Boundary value problems and singular
pseudo-differential operators}, Pure and Applied Mathematics.
John Wiley \& Sons, Ltd., Chichester, 1998.

\bibitem{SeeleyComplex}
R.~Seeley, \emph{Complex powers of an elliptic operator},
Singular Integrals (Proc.~Sympos.~Pure Math., Chicago, Ill., 1966),
pp.~288--307, Amer.~Math.~Soc., 1967.

\bibitem{SeeleyResBVP}
\bysame, \emph{The resolvent of an elliptic boundary problem},
Amer. J. Math. \textbf{91} (1969), 889--920. 

\bibitem{Shubin}
M.~Shubin, \emph{Pseudodifferential operators and spectral 
theory}, Springer Verlag, Princeton, N.J., 1987.

\end{thebibliography}
\end{document}